\documentclass{article}

\usepackage{arxiv}

\usepackage[utf8]{inputenc} 
\usepackage[T1]{fontenc}    
\usepackage{hyperref}       
\usepackage{url}            
\usepackage{booktabs}       
\usepackage{amsfonts}       
\usepackage{nicefrac}       
\usepackage{microtype}      
\usepackage{lipsum}		
\usepackage{graphicx}
\usepackage{amsmath}
\usepackage{amssymb}
\usepackage{color}

\newtheorem{defi}{Definition}
\newcommand{\brdef}{\begin{defi}}
	\newcommand{\erdef}{\end{defi}}
\newtheorem{cor}{Corollary}
\newcommand{\bcor}{\begin{cor}}
	\newcommand{\ecor}{\end{cor}}
\newtheorem{proof}{Proof}
\newcommand{\bproof}{\begin{Proof}}
\newcommand{\eproof}{\end{Proof}}
\newtheorem{thm}{Theorem}
\newtheorem{lem}{Lemma}

\newcommand{\ble}{\begin{lem}}
	\newcommand{\ele}{\end{lem}}

\title{Special Finsler spaces admitting a semi-concurrent vector field}

\author{ M. R. Rajeshwari$^{*}$, S. K. Narasimhamurthy$^{**}$ and H. M. Manjunatha$^{***}$ \\
	Department of PG Studies and Research in Mathematics\\
	Kuvempu University, Jnana Sahyadri\\
	Shankaraghatta - 577 451, Shivamogga, Karnataka, India \\
	$^{*}$\texttt{rajeshwarimr27@gmail.com} \\
	$^{**}$\texttt{nmurthysk@gmail.com} \\
	$^{***}$\texttt{manjunathahmnmt@gmail.com} \\
}

\date{}

\hypersetup{
	pdftitle={Special Finsler spaces admitting a semi-concurrent vector field},
	pdfsubject={math.DG},
	pdfauthor={M. R. Rajeshwari, S. K. Narasimhamurthy, H. M. Manjunatha},
	pdfkeywords={Semi-concurrent vector fields, Special Finsler spaces, $C$-conformal condition},
}

\begin{document}
	\maketitle
	
	\begin{abstract}
		The main objective of this paper is to study semi-concurrent vector fields on a Finsler manifold. We show that the quasi-$C$-reducible Finsler space, $C3$-like Finsler space, $C^{h}$-recurrent Finsler space,  and $P2$-like Finsler space are equivalent to Riemannian if they admit a semi-concurrent vector field. Further, we prove the necessary and sufficient condition for a Finsler space satisfying $C$-conformal condition to become Riemannian.
	\end{abstract}

	
\section{Introduction}
\par In the Finsler geometry, all geometrical quantities are direction-dependent. This feature makes Finsler geometry more suitable for studying physical theories like particle physics, relativistic optics, and general relativity. Finsler geometry has Riemannian geometry as a particular case. Recently, Finsler geometry was applied to study various astrophysical topics like wormhole models \cite{Manjunatha,Rahaman,Sanjay}, compact star models \cite{Rahaman2}, gravastar models \cite{Banerjee,Sanjay2} and so on. The obtained results agree well with the accuracy of experimental observations. Nekouee et al. \cite{Nekouee} have investigated the applications of the Finsler-Randers metric in
cosmology. The Finslerian Schwarzschild-de sitter space-time is investigated in \cite{Manjunatha2}. So, many researchers have been fascinated by this generalized geometry and studying the geometry of a Finsler manifold.

\par Brickell and Yano \cite{e2} studied the concept of concurrent vector fields on a Riemannian manifold. In 1950, Tachibana \cite{e17} defined and discussed the concurrent vector fields. Later, in 1974, Matsumoto and Eguchi \cite{e8} extended the study of concurrent vector fields. In 2004, Rastogi and Dwivedi \cite{e16} discussed the concurrent vector fields on a Finsler manifold.

\par Semi-concurrent vector fields are studied on a Finsler manifold in \cite{e19}. And we found from \cite{e19} that some special Finsler spaces admitting semi-concurrent vector fields reduce to Riemannian. Further, they \cite{e19} discussed some properties of concurrent vector fields. Youssef \textit{et al.} \cite{e18} studied the concurrent $\pi$-vector fields on some special Finsler spaces. And found that the special Finsler spaces admitting a concurrent $\pi$- vector field reduce to a Riemannian.

\par In this paper, we study some results on the Finsler spaces satisfying the $C$-conformal condition. Also, we discuss some special Finsler spaces admitting a semi-concurrent vector field.
	
\section{Notations and preliminaries}
\par Let $(M, F)$ be an $n$-dimensional smooth connected Finsler manifold; $F$ being the Finsler function. Let $(x^i)$ be the coordinates of any point of the base manifold $M$ and $(y^i)$ a supporting element at the same point. We use the following terminology and notations:\\
$\partial_{i}$: partial differentiation with respect to $x^i$,\\
$\dot{\partial}_{i}$: partial differentiation with respect to $y^i$ (basis vector fields of the vertical bundle),\\
$g_{ij}:= \frac{1}{2}\dot{\partial}_{i}\dot{\partial}_{j}F^2 = \dot{\partial}_{i}\dot{\partial}_{j}E$: the Finsler metric tensor, where $E = \frac{1}{2}F^2$ is the energy function,\\
$l_{i}:= \dot{\partial}_{i}F = g_{ij}l^{j} = g_{ij}\frac{y^j}{F}$: the normalized supporting element; $l^{i}:= \frac{y^i}{F}$,\\
$l_{ij}:= \dot{\partial}_{i}l_{j}$,\\
$h_{ij}:= Fl_{ij} = g_{ij} - l_{i}l_{j}$: the angular metric tensor,\\
$C_{ijk}:= \frac{1}{2}\dot{\partial}_{k}g_{ij} = \frac{1}{4}\dot{\partial}_{i}\dot{\partial}_{j}\dot{\partial}_{k}F^2$: the Cartan tensor,\\
$C^{i}_{jk}:= g^{ri}C_{rjk} = \frac{1}{2}g^{ri}\dot{\partial}_{k}g_{rj}$: the $h(hv)$-torsion tensor,\\
$G^{i}$:= the components of the geodesic spray associated with $(M, F)$,\\
$N^{i}_{j}:= \dot{\partial}_{j}G^{i}$: the Barthel (or Cartan nonlinear) connection associated with $(M, F)$,\\
$\delta_{i}:= \partial_{i} - N^{r}_{i}\dot{\partial}_{r}$: the basis vector fields of the horizontal bundle,\\
$G^{i}_{jh}:= \dot{\partial}_{h}N^{i}_{j} = \dot{\partial}_{h}\dot{\partial}_{j}G^{i}$: the coefficients of Berwald connection,\\
$\Gamma^{i}_{jk}:= \frac{1}{2}g^{ir}(\delta_{j}g_{kr} + \delta_{k}g_{jr} - \delta_{r}g_{jk})$: the Christoffel symbols with respect to $\delta_{i}$,\\
$(\Gamma^{i}_{jk}, N^{i}_{j}, C^{i}_{jk})$: the Cartan connection.

\par For the Cartan connection $(\Gamma^{i}_{jk}, N^{i}_{j}, C^{i}_{jk})$, we define\\
$X^{i}_{j|k}:= \delta_{k}X^{i}_{j} + X^{m}_{j}\Gamma^{i}_{mk} - X^{i}_{m}\Gamma^{m}_{jk}$: the horizontal covariant derivative of $X^{i}_{j}$,\\
$X^{i}_{j}|k:= \dot{\partial}_{k}X^{i}_{j} + X^{m}_{j}C^{i}_{mk} - X^{i}_{m}C^{m}_{jk}$: the vertical covariant derivative of $X^{i}_{j}$.

\begin{defi}
	\cite{e1} A Finsler structure on a manifold $M$ is a function
	\begin{equation*}
	F:TM\longrightarrow [0, \infty)
	\end{equation*}
	with the following properties:
\begin{enumerate}
	\item $F$ is $C^\infty$ on the slit tangent bundle $\mathcal{T}M:=TM\setminus{0}$.
	\item $F(x, y)$ is positively homogenous of degree one in $y:F(x, \lambda y) = \lambda F(x, y)$ for all $y\in TM$ and $\lambda > 0$.
	\item The Hessian matrix $g_{ij}(x, y):=\dot{\partial}_{i}\dot{\partial}_{j}E$ is positive-definite at each point of $\mathcal{T}M$, where $E:= \frac{1}{2}F^2$ is the energy function of the Lagrangian $F$.
\end{enumerate}
The pair $(M, F)$ is called a Finsler manifold and the symmetric bilinear form $g = g_{ij}(x, y)dx^{i}\bigotimes dx^{j}$ is called the Finsler metric tensor of the Finsler manifold $(M, F)$.\\
Sometimes, a function $F$ satisfying the above conditions is said to be a regular Finsler metric.
\end{defi}
Next, we will provide the definitions of some special Finsler spaces.
\begin{defi}\cite{e6}
	A Finsler manifold $(M, F)$ of dimension $n \geq 3$ is called a $C$-reducible if the Cartan tensor $C_{ijk}$ has the following form:
	\begin{equation}\label{e:e1}
	C_{ijk} = \frac{1}{n+1}(h_{ij}C_{k} + h_{ki}C_{j} + h_{jk}C_{i}).
	\end{equation}
\end{defi}
\begin{defi}\cite{e9}
	A Finsler manifold $(M,F)$ of dimension $n \geq 3$ is called a semi-$C$-reducible if the Cartan tensor $C_{ijk}$ is written in the following form:
	\begin{equation}\label{e:e2}
	C_{ijk} = \frac{r}{n+1}(h_{ij}C_{k} + h_{ki}C_{j} + h_{jk}C_{i}) + \frac{t}{C^{2}}C_{i}C_{j}C_{k},
	\end{equation}
	where $r$ and $t$ are scalar functions such that $r+t=1$.
\end{defi}
\begin{defi}\cite{e13}
	A Finsler manifold $(M,F)$ of dimension $n \geq 3$ is called a quasi-$C$-reducible if the Cartan tensor $C_{ijk}$ has the following form:
	\begin{equation}\label{e:e3}
	C_{ijk} = h_{ij}C_{k} + h_{ki}C_{j} + h_{jk}C_{i}.
	\end{equation}
\end{defi}
In a three-dimensional Finsler space, $F^3$, $C_{ijk}$ is always written in the form
\begin{equation}\label{e:e4}
LC_{ijk} = Hm_{i}m_{j}m_{k} - J \mathcal{C}_{ijk}(m_{i}m_{j}m_{k}) + I\mathcal{C}_{ijk}(m_{i}m_{j}m_{k}) + Jn_{i}n_{j}n_{k}
\end{equation}
where $L=L(x, y)$ is the fundamental function, $\mathcal{C}_{(ijk)}\left\lbrace \right\rbrace $ denote the cyclic permutation of indices $i, j, k$ and addition. $H, I$ and $J$ are main scalars and $(I_{i}, m_{i}, n_{i})$ is Moor's frame \cite{e6,e13,e10}. Here $I_{i} = \dot{\partial}_{i}L$ is the unit vector along the element of support, $m_{i}$ is the unit vector along $C_{i}$, that is, $m_{i} = C_{i}/C$, where $C^2 = g^{ij}C_{i}C_{j}$ and $n_{i}$ is a unit vector orthogonal to the vectors $I_{i}$ and $m_{i}$. Since the angular metric tensor $h_{ij}$ in $F^3$ can be written as
\begin{equation}\label{e:e5}
h_{ij} = m_{i}m_{j} + n_{i}n_{j},
\end{equation}
we may write \eqref{e:e4} as
\begin{equation}\label{e:e6}
C_{ijk} = \mathfrak{C}_{(ijk)}\left\lbrace h_{ij}a_{k} + C_{i}C_{j}b_{k}\right\rbrace .
\end{equation}
where $a_{k} = \frac{1}{L}(Im_{k} + \frac{J}{3}n_{k})$ and $b_{k} = \frac{1}{LC^2}\left\lbrace (\frac{H}{3} - I)m_{k} + \frac{4J}{3}n_{k}\right\rbrace$. Many authors \cite{e8,e6,e9,e14} have obtained various interesting special forms of $C_{ijk}$.
\begin{defi}\cite{e15}
	A Finsler space $F^{n}$ of dimension $n \geq 4$ is called a $C3$-like Finsler space if there exist covariant vector fields $a_{k}$ and $b_{k}$ in $F^{n}$ such that its $(h)$ $hv$-torsion tensor $C_{ijk}$ can be written as
	\begin{equation}\label{e:e7}
	C_{ijk} = \mathfrak{C}_{(ijk)}\left\lbrace h_{ij}a_{k} + C_{i}C_{j}b_{k}\right\rbrace.
	\end{equation}
	Since $C_{ijk}$ is an indicatory tensor, it follows that $a_{k}$ and $b_{k}$ are indicatory tensors, that is, $a_{0} = b_{0} = 0$. If $b_{k}$ is a null vector then contracting \eqref{e:e7} with $g^{ik}$ and putting $b_{k} = 0$, we get $a_{k} = \frac{1}{n + 1}C_{k}$.
\end{defi}

\begin{defi}\cite{e1}
	A tangent vector field $X$ of a Finsler space $F^{n}$ is concurrent under the Cartan connection, if
	\begin{equation}\label{e:e8}
	X^{i}_{|j} = \partial_{j}X^{i} - N^{h}_{j}\dot{\partial}_{h}X^{i} + X^{h}F^{i}_{hj} = \partial_{j}X^{i} + X^{h}F^{i}_{hj} = -\delta^{i}_{j}
	\end{equation}
	\begin{equation}\label{e:e9}
	X^{i}_{|j} = \dot{\partial}_{j}X^{i} + X^{h}C^{i}_{hj} = X^{h}C^{i}_{hj} = 0.
	\end{equation}
	where $\partial_{j}$ and $\dot{\partial}_{j}$ denote partial differentiations by $X^{j}$ and $Y^{j}$ respectively.
\end{defi}
The Ricci identities \cite{e7} in a Finsler space are as follows:
\begin{equation}\label{e:e10}
X^{h}R_{hijk} = 0,
\end{equation}
\begin{equation}\label{e:e11}
X^{h}P_{hijk} + C_{ijk} = 0,
\end{equation}
\begin{equation}\label{e:e12}
X^{h}S_{hijk} = 0,
\end{equation}
where $R_{hijk}$, $P_{hijk}$ and $S_{hijk}$ are the components of the curvature tensors of $C\Gamma$. Since $P_{hijk}$ are skew symmetric in $h$ and $i$, we have from \eqref{e:e9} that
\begin{equation}\label{e:e13}
X^{h}C_{ijk} = 0.
\end{equation}
We know that the well-known identity,
\begin{equation}\label{e:e14}
P_{hijk} = C_{ijk|h} - C_{hjk|i} + C_{hjr}C^{r}_{ik|0} - C_{ijr}C^{r}_{hk|0}.
\end{equation}
Substitute \eqref{e:e14} in \eqref{e:e9}, we get
\begin{equation}\label{e:e15}
X^{h}(P_{hijk} - C_{ijk|h} + C_{hjk|i} - C_{hjr}C^{r}_{ik|0} + C_{ijr}C^{r}_{hk|0}) C^{i}_{hj} = 0.
\end{equation}
We know that,
\begin{equation}\label{e:e16}
C_{ijk|h} = C^{r}_{jk|0} - C^{r}_{ih|0} + C_{kir}C^{r}_{jh|0} - C_{ijr}C^{r}_{hk|0} - P_{ijkh}.
\end{equation}
Contracting the above equation, we get
\begin{equation}\label{e:e17}
X^{h}C_{ijk|h} = 0.
\end{equation}

\begin{defi}\cite{e4}
	A Finsler space $F^{n}$ is called a $C^{h}$-recurrent if the torsion tensor satisfies the following equation:
	\begin{equation}\label{e:e18}
	C_{ijk|h} = C_{ijk}K_{h},
	\end{equation}
	where $K_{h}$ is a covariant vector field.
\end{defi}

\begin{defi}\cite{e4}
	For a $P2$-like Finsler space $F^{n}$, we have
	\begin{equation}\label{e:e19}
	P_{hijk} = K_{h}C_{ijk} - K_{i}C_{kjh},
	\end{equation}
	where $K_{h}$ is a covariant vector field.
\end{defi}

\begin{defi}\cite{e4}
	A Finsler space $F^{n}$ is called a $P$-reducible if the torsion tensor $P_{ijk}$ is defined as follows:
	\begin{equation}\label{e:e20}
	P_{ijk} = \frac{1}{n+1}(h_{ij}P_{k} + h_{jk}P_{i} + h_{ki}P_{j}),
	\end{equation}
	where $P_{i} = P^{r}_{ir} = C_{i}$.
\end{defi}

\begin{defi}\cite{e4}
	A Finsler space $F^{n}$ is said to satisfy $T$-condition if the $T$-tensor $T_{hijk}$ vanishes identically, that is,
	\begin{equation}\label{e:e22}
	T_{hijk} = LC_{hij|k} + l_{h}C_{ijk} + l_{i}C_{hjk} + l_{j}C_{hik} + l_{k}C_{hij}=0.
	\end{equation}
\end{defi}

\section{Semi-concurrent vector fields}
Let $(M, F)$ be an $n$-dimensional smooth Finsler manifold.
\begin{defi}\cite{e17}
A vector field $X^i(x)$ on $M$ is said to be concurrent under the Cartan connection if it satisfies the following expression:
\begin{equation}\label{e:e23}
X^{h}(x)C_{hij} = 0,\,\,\,\,X^{i}_{|j} = -\delta^{i}_{j}.
\end{equation}
The condition \eqref{e:e23} is called the $C$-condition.
\end{defi}
\begin{defi}\cite{e6}
	The manifold $M$ fulfills the $C$-conformal condition if there exists a conformal transformation $\bar{F} = e^{\sigma(x)}F$ on $M$ such that
	\begin{equation}\label{e:e24}
	\sigma_{h}(x)C^{h}_{ij} = 0,
	\end{equation}
	where $\sigma_{h}:= \frac{\partial \sigma}{\partial x^{h}}$. The condition \eqref{e:e24} will be called the $CC$-condition.
\end{defi}
\begin{defi}\cite{e19}
	A vector field $B^{i}(x)$ on $M$ is said to be semi-concurrent if it satisfies the following expression:
	\begin{equation}\label{e:e25}
	B^{h}(x)C_{hij} = 0.
	\end{equation}
	The condition \eqref{e:e25} will be called the $SC$-condition.
\end{defi}

\begin{lem}\cite{e19}
		For the nonzero function $B^i$ satisfying \eqref{e:e25}, if the scalars $\alpha$ and $\alpha'$ satisfy
	\begin{equation}\label{e:e26}
	\alpha B^i + \alpha'y^i = 0,
	\end{equation}
	then $\alpha = \alpha' = 0$, which means that the two vector fields $B^i(x)$ and $y^i$ are independent.
\end{lem}

\begin{thm}
	A quasi-$C$-reducible Finsler space $(M, F)$ admitting a semi-concurrent vector field is Riemannian.
\end{thm}
\begin{proof}
For a quasi-$C$-reducible Finsler space, we have
\begin{equation}\label{e:e27}
C_{ijk} = h_{ij}C_{k} + h_{jk}C_{i} + h_{ki}C_{j}.
\end{equation}
Contracting \eqref{e:e27} by $B^iB^j$ and by using (\ref{e:e25}), we get
\begin{equation}\label{e:e29}
B^iB^jh_{ij}C_{k} = 0.
\end{equation}
This implies that $C_{k} = 0$. Hence the space is Riemannian.
\end{proof}

\begin{thm}
	A $C3$-like Finsler space satisfying the $SC$-condition is Riemannian provided $J=0$.
\end{thm}
\begin{proof}
	For the $C_{3}$-like Finsler space, we have
	\begin{equation}\label{e:e30}
	C_{ijk} = h_{ij}a_{k} + C_{i}C_{j}b_{k} + h_{jk}a_{i} + C_{j}C_{k}b_{i} + h_{ki}a_{j} + C_{k}C_{i}b_{j}.
	\end{equation}
Contracting \eqref{e:e30} by $B^iB^j$ and from \eqref{e:e25}, we get
\begin{equation}\label{e:e31}
B^iB^j(h_{ij}a_{k} + h_{jk}a_{i} + h_{ki}a_{j}) = 0.
\end{equation}
Equation (\ref{e:e31}) implies that
\begin{equation}\label{e:e36}
B^iB^jn_{i}n_{j}a_{k} + B^jB^kn_{j}n_{k}a_{i} + B^kB^in_{k}n_{i}a_{j} = 0,
\end{equation}
which yields
\begin{equation}\label{e:1}
    \frac{B^{2}I}{LC}C_{k}+\frac{B^{2}J}{L}n_{k}=0.
\end{equation}
In view of (\ref{e:1}), we obtain that $C_{k}=0$ if $J=0$.
\end{proof}
%
%

\begin{thm}
	A $C^{h}$-recurrent Finsler space admitting a semi-concurrent vector field is Riemannian if $1+B^{h}K_{h} \neq 0$.
\end{thm}
\begin{proof}
	For the $C^{h}$-recurrent Finsler space, we have
	\begin{equation}\label{e:e35}
	C_{ijk|h} = C_{ijk}K_{h},
	\end{equation}
	where $K_{h}$ is a covariant vector field.
Consider
\begin{equation}\label{e:2}
P_{hijk} = C_{ijk|h}-C_{hjk|i}+C_{hjr}C^{r}_{ik|0}-C_{ijr}C^{r}_{hk|0}.
\end{equation}
Contracting (\ref{e:2}) by $B^{h}$ yields
\begin{equation}\label{e:3}
B^{h}P_{hijk} = B^{h}C_{ijk|h}.
\end{equation}
From Ricci identity, (\ref{e:3}) implies that
\begin{equation}\label{e:4}
B^{h}C_{ijk|h} = -C_{ijk}.
\end{equation}
On contracting (\ref{e:e35}) by $B^{h}$ and by using (\ref{e:4}), we conclude that $C_{ijk}=0$ and hence the space is Riemannian if $1+B^{h}K_{h}\neq0$.
\end{proof}


\begin{thm}
	A $P2$-like Finsler space admitting a semi-concurrent vector field is Riemannian provided $1+B^{h}K_{h}\neq0$.
\end{thm}
\begin{proof}
A $P2$-like Finsler space is characterized by
\begin{equation}\label{e:e41}
P_{hijk} = K_{h}C_{ijk} - K_{i}C_{kjh},
\end{equation}
where $K^{h}$ is a covariant vector field.
Consider
\begin{equation}\label{e:e43}
P_{hijk} = C_{ijk|h} - C_{hjk|i} + C_{hjr}C^{r}_{ik|0} - C_{ijr}C^{r}_{hk|0}.
\end{equation}
Contracting (\ref{e:e43}) by $B^{h}$, we obtain
\begin{equation}\label{e:e44}
B^{h}P_{hijk} = B^{h}C_{ijk|h}.
\end{equation}
From Ricci identity, (\ref{e:e44}) yields
\begin{equation}\label{e:e441}
B^{h}C_{ijk|h} = -C_{ijk},
\end{equation}
On contracting (\ref{e:e41}) by $B^{h}$ and by using (\ref{e:e441}), we conclude that $C_{ijk} = 0$ if $1 + B^{h}K_{h} \neq 0$, which implies that the space is Riemannian.
\end{proof}

\begin{thm}
	A P-reducible Finsler space $F^{n}$ admitting a semi-concurrent vector field is a Landsberg space provided $B^{2}F^{2}-B^{2}_{0} \neq 0$.
\end{thm}
	\begin{proof}
		The torsion tensor $P_{ijk}$ is given by
		\begin{equation}\label{e:e48}
		P_{ijk} = h_{ij}P_{k} + h_{jk}P_{i} + h_{ki}P_{j}.
		\end{equation}
		Contracting (\ref{e:e48}) by $B^{i}B^{j}$ and since $P_{ijk} = C_{ijk}$, we obtain
		\begin{equation}\label{e:e49}
		B^{i}B^{j}h_{ij}C_{k}\,=\,0,
		\end{equation}
		which implies that
        \begin{equation}\label{e:e481}
		(B^{2}F^{2} - B^{2}_{0})C_{k} = 0.
		\end{equation}
		From (\ref{e:e481}), we deduce that $C_{k} = 0$, that is, $P_{k} = 0$ if $B^{2}F^{2} - B^{2}_{0}\neq0$, which implies that $F^{n}$ is a Landsberg space.
	\end{proof}
The $T$-tensor is defined as follows \cite{e11}:
\begin{equation}\label{e:e482}
T_{hijk} = FC_{hij}|_{k} + C_{hij}l_{k} + C_{hik}l_{j} + C_{hjk}l_{i} + C_{ijk}l_{h}
\end{equation}
If $(M, F)$ is Riemannian, then the $T$-tensor vanishes. If $(M, F)$ satisfies the $CC$-condition, then the converse part of the result holds, which is shown in the following theorem.
\begin{thm}
	A Finsler manifold satisfying the $CC$-condition is Riemannian if and only if the $T$-tensor $T_{hijk}$ vanishes.
\end{thm}
\begin{proof}
	First, we prove that the vertical covariant derivative of $\sigma_{h}$ vanishes identically. Since $\sigma_{h}=\sigma_{h}(x)$, we have
	\begin{equation*}
	\sigma_{h}|_{k} = \dot{\partial}_{k}\sigma_{h} + \sigma_{i}C^{i}_{hk} = 0.
	\end{equation*}
	Let the $T$-tensor vanish, then from (\ref{e:e482}), we have
	\begin{equation*}
	FC_{hij}|_{k} + C_{hij}l_{k} + C_{hik}l_{j} + C_{hjk}l_{i} + C_{ijk}l_{h} = 0.
	\end{equation*}
	Contracting by $\sigma_{m}$, and taking into account that $\sigma_{m}|_{k} = 0$, we get
	\begin{equation*}
	\sigma_{m}C_{hij}l_{k} + \sigma_{m}C_{hik}l_{j} + \sigma_{m}C_{hjk}l_{i} + \sigma_{m}C_{ijk}l_{h} = 0.
	\end{equation*}
	Again contracting by $g^{mh}$, we obtain
\begin{equation*}
\frac{\sigma_{0}}{F}C_{ijk} = 0.
\end{equation*}
Since $\sigma_{0} \neq 0$, it follows that $C_{ijk} = 0$.
\end{proof}
By defining the tensor $T_{ij}$,
\begin{equation*}
T_{ij} := T_{ijhk}g^{hk} = FC_{i}|_{j} + l_{i}C_{j} + l_{j}C_{i},
\end{equation*}
we can state the following result.

\begin{cor}
	A Finsler manifold satisfying the $CC$-condition is Riemannian if and only if the tensor $T_{ij}$ vanishes.
\end{cor}

\section{Conclusions}
In Finsler geometry, special Finsler spaces play a significant role. In this context, we have studied semi-concurrent vector fields on some special Finsler spaces and obtained interesting results. We have shown that a quasi-$C$-reducible Finsler space, $C3$-like Finsler space, $C^{h}$-recurrent Finsler space, and $P2$-like Finsler space admitting a semi-concurrent vector field become Riemannian. And a $P$-reducible Finsler space satisfying the $SC$-condition becomes a Landsberg space. Further, we have obtained the necessary and sufficient condition for a Finsler space satisfying the $C$-conformal condition to become Riemannian.

\end{document}